\documentclass[12pt]{article}
\usepackage{amsmath}
\usepackage{amsthm}
\usepackage{amsfonts}
\usepackage{amssymb}
\usepackage{hyperref}

\newtheorem{thm}{Theorem}[section]
\newtheorem{lem}[thm]{Lemma}

\newtheorem{cor}[thm]{Corollary}

\theoremstyle{definition}

\let\al=\alpha

\let\lm=\lambda

\let\vph=\varphi
\let\abs=\envert

\let\abs=\envert
\newcommand{\floor}[1]{\left\lfloor#1\right\rfloor}

\theoremstyle{remark}

\begin{document}
\title{Explicit formulae for primes in arithmetic progressions, I
}
\author{Tomohiro Yamada}
\date{}
\maketitle

\begin{abstract}
We shall give an explicit formula for $\psi(x, q, a)$ with an error term of the form $C/\log^\alpha x$
under the condition that $q<\log^{\alpha_1} x$ is nonexceptional, for various values of $\alpha$ and $\alpha_1$.
We shall also give an explicit formula for $\psi(x, q, a)$ with error terms $C/\log^A x$
working whether $q$ is exceptional or nonexceptional, but under the condition that
$\frac{0.4923A}{\pi}q^{1/2}\log^2 q<\log x/\log\log x$.
Moreover, we shall give an explicit form of Bombieri-Vinogradov theorem over non-exceptional moduli.
\end{abstract}

\section{Introduction}\label{intro}
The prime number theorem for arithmetic progressions states that $\psi(x, q, a)=\sum_{p^e\leq x, p^e\equiv a\pmod{Q}}\log p$
is $x/\vph(k)+o(x)$ as $x\rightarrow\infty$.

Moreover, Siegel-Walfisz theorem states that
for any $\alpha>0$, the error term can be bounded by $O(x\exp-\frac{c\log^{3/5} x}{(\log\log x)^{1/5}})$
uniformly for $q<\log^\alpha x$.

Prime number theorem for arithmetic progressions has many applications in number theory.
Naturally we need have explicit estimates for the error term for the ordinary prime number theorem
to give explicit results in such applications.

The classical results are Rosser\cite{Ros} and Rosser and Schoenfeld\cite{RS1}\cite{RS2},
giving explicit estimates for the ordinary prime number theorem (i.e. the case $q=1, 2$).
It is after 1980's that various explicit formulae for primes in arithmetic progression have been obtained
by several authors such as McCurley\cite{MC1, MC2, MC3}, Ramar\'{e} and Rumely\cite{RR}, Dusart\cite{Dus},
Liu and Wang\cite{LW} and Kadiri\cite{Kad}.

The purpose of this paper is to give an explicit formula with an error term of the form $C/\log^{\al_2} x$
under the condition $q<\log^{\al_1} x$ for various values of $\al_1$ and $\al_2$.  Our main results
are the following two theorems.

\begin{thm}\label{thm11}
Let $\al_1, \al_2, Y_0=\log\log X_0, C$ be constants given in some column of Table \ref{tbl4}.
Let $q$ be a modulus $\leq \log^{\al_1} x$.
Let $E_0=1$ and $\beta_0$ denote the Siegel zero modulo $q$ if it exists and $E_0=0$ otherwise. 
If $(a, q)=1$ and $x\geq X_0$, then
\begin{equation}
-1+x^{-1}\sum_{\chi\pmod{q}}\abs{\psi(x, \chi)}<\frac{C}{\log^{\al_2} x}+E_0\frac{x^{\beta_0-1}}{\beta_0}
\end{equation}
and
\begin{equation}
\frac{\vph(q)}{x}\abs{\psi(x; q, a)-\frac{x}{\vph(q)}}<\frac{C}{\log^{\al_2} x}+E_0\frac{x^{\beta_0-1}}{\beta_0}.
\end{equation}
\end{thm}

\begin{thm}\label{thm12}
Let $A, Y_0=\log\log X_0, C$ be constants given in some column of Table \ref{tbl4}
with $(\alpha_1, \alpha_2)=(1, A)$ and $C^\prime=C+1/(1-AY_0/e^{Y_0})$.
If $x>X_0$ and $\frac{0.4923A}{\pi}q^{1/2}\log^2 q<\log x/\log\log x$,
then we have
\begin{equation}
-1+x^{-1}\sum_{\chi\pmod{q}}\abs{\psi(x, \chi)}<\frac{C^\prime}{\log^A x}
\end{equation}
and
\begin{equation}
\frac{\vph(q)}{x}\abs{\psi(x; q, a)-\frac{x}{\vph(q)}}<\frac{C^\prime}{\log^A x}.
\end{equation}
\end{thm}

One of applications of our formula is obtaining (almost) completely explicit Bombieri-Vinogradov theorem.

The well-known Bombieri-Vinogradov theorem, first proved by Bombieri\cite{Bom},
states that for any constant $A>0$, there exists some constant $B$ such that
\begin{equation}
\sum_{q\leq \frac{x^\frac{1}{2}}{\log^B x}}\max_{(a, q)=1, 2\leq y\leq x}\abs{\psi(y; q, a)-\frac{\psi(y)}{\vph(q)}}=O\left(\frac{x}{\log^A x}\right),
\end{equation}
where $\pi(y; q, a)$ denotes the number of primes $\leq y$ congruent to $a\pmod{q}$.
Akbary and Hambrook\cite{AH} gave an explicit formula with $B=A+\frac{9}{2}$,
with moduli limited to those which is free of small prime factors.  Our explicit formula
enables us to give (almost) completely explicit version of Bombieri-Vinogradov theorem.

Before stating our version of Bombieri-Vinogradov theorem, we introduce
a result concerning zeros of Dirichlet L-functions.

\begin{lem}\label{lem10}
Define $\Pi(s, q)=\prod_{\chi\pmod{q}}L(s, \chi), R_0=6.3970$ and $R_1=2.0452$.
Then the function $\Pi(s, q)$ has at most one zero $\rho=\beta+it$ in the region
$0\leq\beta\leq 1-1/R_0\log \max\{q, q\abs{t}\}$.  If such a zero exists,
then it must be real and simple and must correspond to a nonprincipal real character $\chi\pmod{q}$.
Moreover, for any given $Q_1$, such a zero must satisfy $\beta<1-1/2R_1\log Q_1$ except
possibly one modulus $\leq Q_1$.
\end{lem}

This result easily follows from Theorems 1.1 and 1.3 of \cite{Kad}.

For a given $Q_1$, we call a modulus $q_0\leq Q_1$ to be exceptional up to $Q_1$ if $\Pi(s, q)$ has a zero $\rho=\beta+it$ with $\beta\geq 1-1/2R_1\log Q_1$.
Now we shall state our explicit version of Bombieri-Vinogradov theorem.

\begin{thm}\label{thm13}
Let $A$ be an integer with $6\leq A\leq 10$ and $Y_0=10.1 (A=6), 10.5 (A=7), 10.9 (A=8), 11.2 (A=9)$ and $11.5 (A=10)$,
$C$ be the constant in Table \ref{tbl4} corresponding to the column $(\alpha_1, \alpha, Y)=(A, A-3, Y_0)$,
$c_0, c_1$ be the constants defined by
\begin{equation}
\begin{split}
c_0= & \frac{2^{\frac{13}{2}}}{9\pi\log 2}\left(\frac{1}{3}+\frac{3}{2\log 2}\right)\left(\frac{2+\log(\log 2/\log(4/3))}{\log 2}\right)\sqrt\frac{\psi(113)}{113}\\
\end{split}
\end{equation}
and
\begin{equation}
c_1=\prod_p\left(1+\frac{1}{p(p-1)}\right)=\frac{\zeta(2)\zeta(3)}{\zeta(6)}.
\end{equation}
We note that $c_0<48.833$ and $c_1<1.9436$.

Let $Q=\frac{x^{1/2}}{\log^A x}$ and $1\leq Q_1\leq\log^A x$.  Let $q_0$ denote the exceptional modulus up to $Q_1$ if it exists.
If $\log\log x>X_0$, then the inequality
\begin{equation}\label{eq11}
\begin{split}
&\sum_{q\leq Q, q_0\nmid q}\max_{a\pmod{q}}\abs{\psi(x, q, a)-\frac{\psi(x)}{\vph(q)}}\\
&\qquad<x^\frac{1}{2}+\frac{c_1c_0(2+e^{-800})x}{\log^{A-\frac{9}{2}} x}+\frac{2c_0c_1x\log^\frac{9}{2} x}{Q_1}+\frac{c_1^2(C_0+e^{-17})x(1+A\log\log x)}{2\log^{A-4} x}.
\end{split}
\end{equation}
holds.
\end{thm}

Our explicit Bombieri-Vinogradov theorem is not completely explicit in two senses:
First, $\max_{2\leq y\leq x}$ should be removed. But in many applications,
this would not cause a serious problem.
Second, we must avoid exceptional moduli.  However, our version would still have several applications.
Indeed, combined with the author's explicit linear sieve \cite{Ymd1}
and a technique which discharges us from large exceptional moduli,
the author obtained explicit version of Chen's celebrated theorem that any sufficiently large even numbers can be written
as the sum of a prime and a product of at most two primes\cite{Ymd2};
the author showed that $\exp\exp 36$ suffices.

For calculations of constants, we used PARI-GP.  Our script is available from
\url{http://tyamada1093.web.fc2.com/math/files/prim0003pari.txt}
and can be used to calculate constants for arbitrary values of $\alpha$ and $\alpha_1$.

\section{Preliminary Lemmas}\label{lems}
We combine the method of McCurley in \cite{MC2} and results of Kadiri concerning the distribution of zeros
of Dirichlet L-functions in \cite{Kad}.  We begin by the following zero-free region.

Let $T=\log^{\al} x$, where $\al=\al_1+\al_2+3\geq 5$.
We let two constants $C_1, C_2$ be given in any column of Table \ref{tblc12}
and put
\begin{equation}
F(T)=\frac{T}{\pi}\log\frac{qT}{2\pi e}, R(T)=C_1\log qT+C_2.
\end{equation}

Then, Theorem 1 of \cite{Tru} gives that $N(\chi, T)\leq F(T)+R(T)$.
\begin{table}
\caption{constants $C_1, C_2$}\label{tblc12}
\begin{center}
\begin{small}
\begin{tabular}{| c | c |}
 \hline
$C_1$ & $C_2$ \\
 \hline
$0.247$ & $9.359$ \\
$0.264$ & $8.049$ \\
$0.281$ & $7.323$ \\
$0.298$ & $6.828$ \\
$0.315$ & $6.455$ \\
$0.332$ & $6.156$ \\
$0.349$ & $5.907$ \\
$0.365$ & $5.694$ \\
$0.382$ & $5.506$ \\
$0.399$ & $5.338$ \\
 \hline
\end{tabular}
\end{small}
\end{center}
\end{table}

For a Dirichlet character $\chi$ modulo $q$, we use $z(\chi)$ to represent the set of zeros $\rho=\beta+i\gamma$
of $L(s, \chi)$ with $\beta\geq 0$ and $\rho\neq 0$.
Moreover, for $H>1$ and $R>0$, we use $z_0(\chi, H, R)$ to represent the set of zeros $\rho=\beta+i\gamma$ of $L(s, \chi)$
with $\frac{1}{2}<\beta\leq 1-1/R\log qH, \abs\gamma\leq H$
and $z_1(\chi, H, R)$ to represent the set of zeros $\rho=\beta+i\gamma$ of $L(s, \chi)$
with $1-1/R\log qH<\beta\leq 1-1/R_0\log qH, \abs\gamma\leq H$.
Of course, GRH implies that these sets would be empty.

We set 
\begin{equation}
\Sigma=\sum_\chi\sum_{\rho\in z(\chi), \abs\gamma\leq T}\frac{x^{\beta-1}}{\abs\rho}.
\end{equation}

We can estimate error terms in the prime number theorem in arithmetic progressions by these sums.

Similarly to Theorem 8 of \cite{LW}, we have the following lemma.
\begin{lem}\label{lem31}
If $t>\exp\exp 6.8$, then we have
\begin{equation}\label{eq31}
\abs{\psi(t, \chi)-\delta(\chi)t}\leq \sum_{\rho\in z(\chi), \abs\gamma\leq T}\frac{x^\beta}{\abs\rho}+1.3833\frac{t\log^2 x}{T},
\end{equation}
where $\delta(\chi)=1$ if $\chi$ is principal and $\delta(\chi)=0$ otherwise.
\end{lem}
\begin{proof}
We shall mainly consider the case $\chi$ is primitive and show that
(\ref{eq31}) holds with $1.3833$ replaced by $1.38329$.
If $\chi\pmod{q}$ is induced by the primitive character $\chi_1\pmod{q_1}$, then
$\abs{\psi(t, \chi)-\psi(t, \chi_1)}\leq \log q/\log 2\leq A\log L/\log 2\leq e^{-20}tL^2/T$
and (\ref{eq31}) holds.
If $\chi=\chi_0$ is the principal character, then the following argument can be proceeded with
$\psi(t, \chi)$ replaced by $\psi(t, \chi)-t$.

We begin by seeing that Lemma 8 of \cite{CW} gives
$\sum_{\abs{\gamma-T}\leq 1}1\leq 3.5(0.5\log q(T+2)+0.59773)\leq C_3\log(qT)-1$,
where $C_3<1.83$, so that there exists a real $T_0$ such that $\abs{T_0-T}\leq 1$ and
\begin{equation}
\frac{1}{\abs{\gamma-T_0}}\leq C_3\log(qT).
\end{equation}
We can confirm that $L$ and $T$ are enough large for us to obtain
\begin{equation}
\psi(t, \chi)=\frac{1}{2\pi i}\int_{b-iT_0}{b+iT_0} -\frac{L^\prime}{L}(s, \chi)\frac{t_0^s}{s}ds+R_6,
\end{equation}
where $t_0=\floor{t+0.5}$, and
\begin{equation}
\psi(2.5, \chi)=\frac{1}{2\pi i}\int_{b-iT_0}{b+iT_0} -\frac{L^\prime}{L}(s, \chi)\frac{2.5^s}{s}ds+R_7,
\end{equation}
with $\abs{R_6}\leq 1.38305tL^2/T$ and $\abs{R_7}\leq 2L/(T-1)\leq e^{-1000} tL^2/T$.
Thus (4.5) in \cite{LW} holds with $\abs{R_8}\leq \abs{R_6}+\abs{R_7}+\log 2< 1.38306tL^2/T$
and (4.6) in \cite{LW} holds with $\abs{R_9}\leq \abs{R_8}+\exp(\log(t+1)/L)TL/(2\pi)\leq 1.38307tL^2/T$.

Using Lemmas 9 and $9^\prime$ of \cite{CW}, we have
\begin{equation}
\abs{\frac{L^\prime}{L}(\sigma+iT,\chi)}\leq 4.6\log(qT)+C_3^2 \log^2 (qT)
\end{equation}
for $-0.5\leq\sigma\leq b$, and
\begin{equation}
\abs{\frac{L^\prime}{L}(-0.5+it,\chi)}\leq 1.5\log q(T+2)+1.79319+\frac{35}{8}\log(T+2)+13.4.
\end{equation}
Thus we obtain
\begin{equation}
\begin{split}
\abs{\frac{1}{2\pi i}\int_{-0.5\pm iT_0}{b\pm iT_0} -\frac{L^\prime}{L}(s, \chi)\frac{t_0^s-2.5^s}{s}ds}
\leq \frac{0.00003tL^2}{T}
\end{split}
\end{equation}
and
\begin{equation}
\abs{\frac{1}{2\pi i}\int_{-0.5-iT_0}{-0.5+iT_0} -\frac{L^\prime}{L}(s, \chi)\frac{t_0^s-2.5^s}{s}ds}.
\leq 1.31\log T\log(qT)
\end{equation}

Hence (4.8) in \cite{LW} holds with $\abs{R_{10}}\leq 1.3831tL^2/T$.
Since $3.5(0.5\log q(T+2)+0.59773)<0.000092(T-1)L^2/T$,
(4.9) in \cite{LW} holds with the constant $1.38039$ replaced by $1.3832$.

By Theorem 1 of \cite{Tru}, we have $N(\chi, 1)\leq 0.399\log q+5.338$ and therefore
\begin{equation}
\begin{split}
\sum_{\abs\gamma\leq 1}\frac{1}{\Re \rho}
& \leq\frac{0.4923q^{\frac{1}{2}}\log^2 q}{\pi}+R_0(\log q)N(\chi, 1) \\
& \leq 0.33q^{\frac{1}{2}}\log^2 q\leq 0.39\al_1^2L^{\frac{\al_1}{2}}\log^2 L \\
& \leq e^{-1000}\frac{tL^2}{T}.
\end{split}
\end{equation}
We see that, using Theorem 6 of \cite{LW}, or Theorem 1 of \cite{Tru}, 
\begin{equation}
\sum_{1\leq \abs\gamma\leq T+1}\frac{1}{\abs\rho}
\leq\frac{N(\chi, T+1)}{T+1}+\int_1^{T+1}\frac{N(\chi, y)}{y^2}dy
\leq\frac{\log^2 (T+1)}{2\pi}
\leq e^{-1000}\frac{tL^2}{T}.
\end{equation}
Hence $\sum_{\abs\gamma\leq T+1}\abs{2.5^\rho/\rho}\leq e^{-999}tL^2/T$.

Similarly, observing that $t^{1-\beta}{1-\beta}$ with $1-\pi/(0.4934q^{1/2}\log^2 q)\leq\beta\leq 1/2$
takes the maximum at $\beta=1/2$, we obtain
\begin{equation}
\begin{split}
\sum_{\abs\gamma\leq 1, \beta\leq 1/2}\frac{t^\frac{1}{2}}{\Re \rho}
& \leq t^{\frac{1}{2}}(2+R_0(\log q)N(\chi, 1)) \\
& \leq 10t^{\frac{1}{2}}\log^2 q\leq 10\al_1^2 t^{\frac{1}{2}}\log^2 L\\
& \leq e^{-1000}\frac{tL^2}{T}
\end{split}
\end{equation}
and
\begin{equation}
\sum_{1\leq \abs\gamma\leq T+1}\frac{1}{\abs\rho}
\leq\frac{N(\chi, T+1)}{T+1}+\int_1^{T+1}\frac{N(\chi, y)}{y^2}dy
\leq\frac{\log^2 (T+1)}{2\pi}<e^{-500}\frac{tL^2}{T}.
\end{equation}

This proves (\ref{eq31}) with $1.3833$ replaced by $1.38329$.
In the imprimitive case, (\ref{eq31}) follows from the argument in the beginning of the proof
and the proof is complete.
\end{proof}

So that, it suffices to obtain upper bounds for $\Sigma$.
The following lemmas, combined with Lemma \ref{lem10}, enable us to obtain these upper bounds.
However, an upper bound for $\Sigma$ obtained by Lemma \ref{lem10} would be large.
We can improve an upper bound by dividing the sum $\Sigma$ as follows.

We additionally set the following two sums
\begin{equation}
\begin{split}
\mu_0=&\sum_\chi\sum_{\rho\in z_0(\chi, T, R)}\frac{x^{\beta-1}}{\abs\rho},\\
\mu_1=&\sum_\chi\sum_{\rho\in z_1(\chi, T, R)}\frac{x^{\beta-1}}{\abs\rho}.
\end{split}
\end{equation}
Then we can easily see that
\begin{equation}\label{eq32}
\Sigma\leq\mu_0+\mu_1+E_0\frac{x^{\beta_0-1}}{\beta_0}.
\end{equation}

Now results in \cite{LW} concerning the number of zeros near the line $\sigma=1$
enable us to improve an upper bound for $\Sigma$ by choosing $R$ larger than $R_0$.

We can easily see that
\begin{equation}
\sum_{\rho\in z_0(\chi, T, R)}\frac{x^{\beta-1}}{\abs\rho}\leq \frac{1}{2}x^{-1/R\log qT}\left(2N(\chi, 1)+\int_1^T\frac{dN(\chi, t)}{t}\right),
\end{equation}
Using to (3.30) in \cite[p. 278]{MC2}, we obtain
\begin{equation}
2N(\chi, 1)+\int_1^T\frac{dN(\chi, t)}{t}\leq \frac{N(\chi, T)}{T}+N(\chi, 1)+\frac{\log^2 T}{2\pi}+F(1)\log T+R(1)+C_1.
\end{equation}
Since $N(T)\leq F(T)+R(T)$ and $N(1)\leq F(1)+R(1)$, this is at most
\begin{equation}
\leq S:=\frac{\al^2+2\al\al_1}{2\pi}l^2+\frac{2\al_1-\al\log (2\pi)}{\pi}l+\frac{R(T)}{T}+2R(1)-\frac{2(1+\log (2\pi))}{\pi}+C_1,
\end{equation}
where $l=\log L=\log\log x$.

Hence we obtain $\mu_0\leq (\vph(q)S)x^{-1/R\log qT}/2$ and using (\ref{eq31}), we conclude that
\begin{equation}\label{eq33}
\sum_{\chi\pmod{q}}\abs{\psi(t, \chi)-\delta(\chi)t}\leq \mu_1+\frac{qS}{2}x^{-1/R\log qT}+1.38\frac{x}{\log^{\al_2+1} x}.
\end{equation}

In order to estimate the remaining term $\mu_1$, we use the following inequality.

\begin{lem}\label{lem34}
Let $\rho=\beta+it$ be a zero of $L(s, \chi)$ with $\beta<1-1/R_0\log qT, \abs{t}\leq T$ and $\rho\neq 0$.
Let $\lm$ be such that $\beta=1-\lm/\log qT$.  Then we have

\begin{equation}\label{eq341}
\abs{\frac{x^{\rho-1}}{\rho}}<\begin{cases}\frac{x^{-1/R_0\log q}}{1-\frac{1}{R_0\log q}} &\text{if }\abs{t}\leq 1,\\
q\exp -2\sqrt\frac{\log x}{R_0} &\text{if } \abs{t}>1, \exp\sqrt\frac{x}{R_0}<(qT)^{1/R_0\lm},\\
q\frac{x^{-\lm/\log qT}}{(qT)^{1/R_0\lm}} &\text{otherwise}.
\end{cases}
\end{equation}
\end{lem}
\begin{proof}
We may assume that $t>0$ and see that $\abs{\frac{x^{\rho-1}}{\rho}}<\frac{x^{\beta-1}}{t}$.
In the case $\abs{t}<1$, 
\begin{equation}
\abs{\frac{x^{\rho-1}}{\rho}}<\frac{x^{\beta-1}}{\beta}.
\end{equation}
Since $\beta>\frac{1}{R_0\log q}\geq\frac{1}{\log x}$, we have (\ref{eq341}) in this case.

If $t\geq 1$, then we have
\begin{equation}
\frac{x^{\beta-1}}{t}<\frac{x^{-1/R_0\log qt}}{t}\leq q\exp -2\sqrt\frac{\log x}{R_0}
\end{equation}
since
\begin{equation}
\frac{1}{R_0\log qt}+\frac{\log t}{\log x}=\frac{1}{R_0\log qt}+\frac{\log qt}{\log x}-\frac{\log q}{\log x}\geq 2\sqrt\frac{1}{R_0\log x}-\frac{\log q}{\log x}.
\end{equation}

Let $t_0=\frac{1}{q}\exp\sqrt\frac{x}{R_0}$ and $t_1=q^{\frac{1}{R_0\lm}-1}T^{\frac{1}{R_0\lm}}$.
Now we may assume that $t_0\geq t_1$, that is, $qt_0=\exp\sqrt\frac{x}{R_0}\geq (qT)^{1/R_0\lm}=qt_1$.
In the case $t\leq t_1$ we have
\begin{equation}
\frac{x^{-1/R_0\log qt}}{t}\leq \frac{x^{-1/R_0\log qt_1}}{t_1}=q\frac{x^{-\lm/\log qT}}{(qT)^{1/R_0\lm}}
\end{equation}
since $\frac{x^{-1/R_0\log qt}}{t}$ is increasing below $t_0$.
Moreover, in the other case we have
\begin{equation}
\frac{x^{-\lm/\log qT}}{t}\leq \frac{x^{-\lm/\log qT}}{t_1}=q\frac{x^{-\lm/\log qT}}{(qT)^{1/R_0\lm}}.
\end{equation}
\end{proof}

\section{Proof of main results}\label{mainresults}
Now we shall estimate $\mu_1$, the sum of $\abs{\frac{x^{\rho-1}}{\rho}}$ over zeros $\rho$ of $L(s, \chi)$ with $1-R<\log qT<\beta_i<1-R_0/\log qT, 1\leq \abs{t_i}\leq T$.

Let $\rho_1=\beta_1+it_1, \rho_2=\beta_2+it_2, \ldots$ be all zeros of $\Pi(s, q)$ with $\beta_i=1-\lm_i/\log qT, 1\leq t_i\leq T$
and $1/R_0<\lm_1\leq\lm_2\leq\lm_3\leq\ldots$.

We shall summarize some results of \cite{LW} concerning the distribution of zeros of $\Pi(s, q)$.

\begin{lem}\label{lem41}
Assume that $qT\geq 8\times 10^9$.  Then $\lm_1\leq \eta_i$ implies $\lm_2>\xi_i$ for each $i=1, 2, \ldots, 12$, where $\eta_i$'s and $\xi_i$'s are given as in Table \ref{tbl1},
and $\lm_3>0.26213$.
Moreover, if $qT\geq 10^{11}$, then
for each $n$ of Table \ref{tbl2}, $\lm_n$ must be greater than the corresponding value,
\end{lem}
\begin{proof}
Theorem 2 and Table 1 of \cite[p.p. 269--270]{LW} give the result for $\lm_1$ and $\lm_2$.
The result for $\lm_3$ is just Theorem 1 of \cite[p. 265]{LW}.
The remaining results follow from Table 3-5 of \cite[p.p. 277--279]{LW}.
\end{proof}

\begin{table}
\caption{constants $\eta_i$ and $\xi_i$}\label{tbl1}
\begin{center}
\begin{small}
\begin{tabular}{| c | c | c |}
 \hline
$i$ & $\eta_i$ & $\xi_i$ \\
 \hline
$1$ & $0.10367089$ & $0.3534$ \\
$2$ & $0.12$ & $0.3221$ \\
$3$ & $0.13$ & $0.3050$ \\
$4$ & $0.14$ & $0.2891$ \\
$5$ & $0.15$ & $0.2743$ \\
$6$ & $0.16$ & $0.2605$ \\
$7$ & $0.17$ & $0.2477$ \\
$8$ & $0.18$ & $0.2356$ \\
$9$ & $0.19$ & $0.2242$ \\
$10$ & $0.20$ & $0.2135$ \\
$11$ & $0.206$ & $0.2074$ \\
$12$ & $0.2067$ & $0.2067$ \\
 \hline
\end{tabular}
\end{small}
\end{center}
\end{table}

\begin{table}
\caption{constants related to each modulus}\label{tbl2}
\begin{center}
\begin{small}
\begin{tabular}{| c | c | c | c | c | c | c | c |}
 \hline
$n$ & $4$ & $5$ & $6$ & $7$ & $10$ & $18$ & $45$ \\
 \hline
$\lm_n>$ & $0.28$ & $0.31$ & $0.32$ & $0.33$ & $0.36$ & $0.39$ & $0.42$ \\
 \hline
\end{tabular}
\end{small}
\begin{small}
\begin{tabular}{| c | c | c | c | c | c |}
 \hline
$n$ & $91$ & $146$ & $332$ & $834$ & $7000$ \\
 \hline
$\lm_n> $ & $0.45$ & $0.46$ & $0.47$ & $0.475$ & $0.478$ \\
 \hline
\end{tabular}
\end{small}
\end{center}
\end{table}

By Lemma \ref{lem34}, we have $\abs{\frac{x^{\rho-1}}{\rho}}<\mu(\lm)$, where $\mu(\lm)$ is the function represented by the right-hand side of (\ref{eq341}).
Now Lemma \ref{lem41} gives an upper bound for $\mu_1$.

\begin{cor}\label{cor42}
Assume that $qT\geq 10^{11}$.  Then, for some $i=1, 2, \ldots, 11$ we have
\begin{equation}
\begin{split}
\mu_1<& 2\min_{0\leq J\leq 12}\left(\mu(\eta_i)+\mu(\xi_{i+1})+\sum_{j=1}^{J} M_j \mu(\max\{\xi_{i+1},\nu_j\})\right)
\end{split}
\end{equation}
with $R=R_{i, J}=\max\{\xi_{i+1}, \nu_{J+1}\}$, where $M_j$'s and $\nu_j$'s are defined in Table \ref{tbl3}.
\end{cor}

So that the sum $\vph(q)Sx^{-1/R\log qT}+\mu_1$ in (\ref{eq33}) is at most
\begin{equation}\label{eq41}
\begin{split}
\max_{1\leq i\leq 11} \min_{0\leq J\leq 12} qSx^{-1/R_{i, J}\log qT}+2\left(\mu(\eta_i)+\mu(\xi_{i+1})+\sum_{j=1}^{J} M_j \mu(\max\{\xi_{i+1},\nu_j\})\right).
\end{split}
\end{equation}

Now we shall prove Theorem \ref{thm11}.  We shall use (\ref{eq41});
we have $qT\geq 10^{11}$ for any value of $\alpha_2$ in Table \ref{tbl4}.
For each pair $(\alpha_1, \alpha_2)$ in Table \ref{tbl4},
we calculated an upper bound for $C$ in appropriate ranges
for each $X_0$ such that $10\log\log X_0$ is an integer and we obtained the desired results
using PARI-GP (The output is available from
\url{http://tyamada1093.web.fc2.com/math/files/prim0003pariout.txt}).
This completes the proof of Theorem \ref{thm11}.

\begin{table}
\caption{constants $\nu_j$ and $M_j$ in Corollary \ref{cor42}}\label{tbl3}
\begin{center}
\begin{small}
\begin{tabular}{| c | c | c | c | c | c | c | c | c |}
 \hline
$j$ & $1$ & $2$ & $3$ & $4$ & $5$ & $6$ & $7$ & $8$ \\
 \hline
$\nu_j$ & $0.26213$ & $0.27$ & $0.30$ & $0.32$ & $0.33$ & $0.36$ & $0.39$ & $0.42$ \\
 \hline
$M_j$ & $1$ & $1$ & $1$ & $1$ & $4$ & $7$ & $27$ & $57$ \\
 \hline
\end{tabular}
\end{small}
\begin{small}
\begin{tabular}{| c | c | c | c | c | c |}
 \hline
$j$ & $9$ & $10$ & $11$ & $12$ & $13$ \\
 \hline
$\nu_j$ & $0.45$ & $0.46$ & $0.47$ & $0.475$ & $0.478$ \\
 \hline
$M_j$ & $55$ & $186$ & $502$ & $6166$ & -- \\
 \hline
\end{tabular}
\end{small}
\end{center}
\end{table}

Next we shall prove Theorem \ref{thm12}.
Since our assumption implies that $q<\log^2 x$, we have by (\ref{eq33}) with $\alpha_1=2$,
\begin{equation}\label{eq51}
-1+x^{-1}\sum_{\chi\pmod{q}}\abs{\psi(x, \chi)}<\frac{C\log\log x}{\log x}+E_0\frac{x^{\beta_0-1}}{\beta_0}.
\end{equation}

Now assume that $E_0=1$ and let $\beta_0>\frac{1}{2}$ be the Siegel zero of $L(s, \chi)$.
Then, Theorem 3 of \cite{LW} gives $\beta_0\leq 1-\frac{\pi}{0.4923q^{1/2}\log^2 q}$.

Since we have assumed that $\frac{0.4923A}{\pi}q^{1/2}\log^2 q<\log x/\log\log x$, we have
\begin{equation}\label{eq52}
\frac{x^{\beta_0-1}}{\beta_0}<\frac{x^{-\frac{\pi}{0.4923q^{1/2}\log^2 q}}}{1-\frac{\pi}{0.4923q^{1/2}\log^2 q}}<\frac{1}{\left(1-\frac{A\log\log x}{\log x}\right)\log^A x}.
\end{equation}

Combining (\ref{eq51}) and (\ref{eq52}), we have
\begin{equation}\label{eq53}
\frac{\vph(q)}{x}\abs{\psi(x; q, a)-\frac{x}{\vph(q)}}<\frac{C^\prime}{\log^A x}
\end{equation}
for $x\geq X_0$ and $\frac{0.4923A}{\pi}q^{1/2}\log^2 q<\log x/\log\log x$.

\begin{table}[t]
\caption{constants related to Theorem 1}\label{tbl4}
\begin{center}
\begin{small}
\begin{tabular}{| c | c | c | l |}
\hline
$\alpha_1$ & $\alpha$ & $Y_0$ & $C$ \\
\hline
$1$ & $1$ & $68$ & $341.37279$ \\
$1$ & $1$ & $69$ & $139.24784$ \\
$1$ & $1$ & $70$ & $51.198745$ \\
$1$ & $1$ & $71$ & $16.805224$ \\
$1$ & $1$ & $72$ & $4.9039512$ \\
$1$ & $1$ & $73$ & $1.2429184$ \\
$1$ & $1$ & $74$ & $0.2720959$ \\
$1$ & $1$ & $75$ & $0.0511499$ \\
$1$ & $1$ & $76$ & $0.0085318$ \\
$1$ & $1$ & $77$ & $0.0016318$ \\
$1$ & $1$ & $78$ & $0.0006712$ \\
$1$ & $1$ & $79$ & $0.0005215$ \\
\hline
$1$ & $2$ & $77$ & $287.04559$ \\
$1$ & $2$ & $78$ & $49.708220$ \\
$1$ & $2$ & $79$ & $7.0904806$ \\
$1$ & $2$ & $80$ & $0.8182991$ \\
$1$ & $2$ & $81$ & $0.0751978$ \\
$1$ & $2$ & $82$ & $0.0056819$ \\
$1$ & $2$ & $83$ & $0.0006284$ \\
$1$ & $2$ & $84$ & $0.0003224$ \\
$1$ & $2$ & $85$ & $0.0002818$ \\
\hline
$1$ & $3$ & $84$ & $27.074957$ \\
$1$ & $3$ & $85$ & $1.4777161$ \\
$1$ & $3$ & $86$ & $0.0591888$ \\
$1$ & $3$ & $87$ & $0.0018985$ \\
$1$ & $3$ & $88$ & $0.0002410$ \\
$1$ & $3$ & $89$ & $0.0001891$ \\
\hline
$1$ & $4$ & $88$ & $169.98979$ \\
$1$ & $4$ & $89$ & $4.3948140$ \\
$1$ & $4$ & $90$ & $0.0767410$ \\
$1$ & $4$ & $91$ & $0.0010204$ \\
$1$ & $4$ & $92$ & $0.0001459$ \\
\hline
$1$ & $5$ & $92$ & $94.371270$ \\
$1$ & $5$ & $93$ & $0.8513896$ \\
$1$ & $5$ & $94$ & $0.0047470$ \\
$1$ & $5$ & $95$ & $0.0001181$ \\
$1$ & $5$ & $96$ & $0.0000938$ \\
\hline
$1$ & $6$ & $95$ & $411.32655$ \\
$1$ & $6$ & $96$ & $1.5600920$ \\
$1$ & $6$ & $97$ & $0.0033364$ \\
$1$ & $6$ & $98$ & $0.0000803$ \\
\hline
\end{tabular}
\begin{tabular}{| c | c | c | l |}
\hline
$\alpha_1$ & $\alpha$ & $Y_0$ & $C$ \\
\hline
$1$ & $7$ & $98$ & $224.55734$ \\
$1$ & $7$ & $99$ & $0.2773700$ \\
$1$ & $7$ & $100$ & $0.0002301$ \\
$1$ & $7$ & $101$ & $0.0000569$ \\
\hline
$1$ & $8$ & $101$ & $8.1198690$ \\
$1$ & $8$ & $102$ & $0.0024266$ \\
$1$ & $8$ & $103$ & $0.0000469$ \\
\hline
$1$ & $9$ & $103$ & $70.055386$ \\
$1$ & $9$ & $104$ & $0.0085547$ \\
$1$ & $9$ & $105$ & $0.0000385$ \\
\hline
$1$ & $10$ & $105$ & $145.53327$ \\
$1$ & $10$ & $106$ & $0.0062008$ \\
$1$ & $10$ & $107$ & $0.0000313$ \\
$1$ & $10$ & $108$ & $0.0000001$ \\
\hline
$2$ & $1$ & $77$ & $366.59584$ \\
$2$ & $1$ & $78$ & $63.466310$ \\
$2$ & $1$ & $79$ & $9.0503268$ \\
$2$ & $1$ & $80$ & $1.0500668$ \\
$2$ & $1$ & $81$ & $0.0967331$ \\
$2$ & $1$ & $82$ & $0.0072533$ \\
$2$ & $1$ & $83$ & $0.0007180$ \\
$2$ & $1$ & $84$ & $0.0003264$ \\
$2$ & $1$ & $85$ & $0.0002820$ \\
\hline
$2$ & $2$ & $84$ & $33.655774$ \\
$2$ & $2$ & $85$ & $1.8367608$ \\
$2$ & $2$ & $86$ & $0.0735210$ \\
$2$ & $2$ & $87$ & $0.0023052$ \\
$2$ & $2$ & $88$ & $0.0002489$ \\
$2$ & $2$ & $89$ & $0.0001892$ \\
\hline
$2$ & $3$ & $88$ & $206.74646$ \\
$2$ & $3$ & $89$ & $5.3442833$ \\
$2$ & $3$ & $90$ & $0.0932719$ \\
$2$ & $3$ & $91$ & $0.0012073$ \\
$2$ & $3$ & $92$ & $0.0001472$ \\
\hline
$2$ & $4$ & $92$ & $112.75512$ \\
$2$ & $4$ & $93$ & $1.0170960$ \\
$2$ & $4$ & $94$ & $0.0056481$ \\
$2$ & $4$ & $95$ & $0.0001209$ \\
$2$ & $4$ & $96$ & $0.0000938$ \\
\hline
\end{tabular}
\end{small}
\end{center}
\end{table}

\begin{table}[t]
\begin{center}
\begin{small}
\begin{tabular}{| c | c | c | l |}
\hline
$\alpha_1$ & $\alpha$ & $Y_0$ & $C$ \\
\hline
$2$ & $5$ & $95$ & $484.26872$ \\
$2$ & $5$ & $96$ & $1.8365549$ \\
$2$ & $5$ & $97$ & $0.0039123$ \\
$2$ & $5$ & $98$ & $0.0000809$ \\
\hline
$2$ & $6$ & $98$ & $261.09757$ \\
$2$ & $6$ & $99$ & $0.3224670$ \\
$2$ & $6$ & $100$ & $0.0002573$ \\
$2$ & $6$ & $101$ & $0.0000569$ \\
\hline
$2$ & $7$ & $101$ & $9.3404567$ \\
$2$ & $7$ & $102$ & $0.0027835$ \\
$2$ & $7$ & $103$ & $0.0000469$ \\
\hline
$2$ & $8$ & $103$ & $79.844778$ \\
$2$ & $8$ & $104$ & $0.0097436$ \\
$2$ & $8$ & $105$ & $0.0000386$ \\
\hline
$2$ & $9$ & $105$ & $164.53130$ \\
$2$ & $9$ & $106$ & $0.0070055$ \\
$2$ & $9$ & $107$ & $0.0000313$ \\
\hline
$2$ & $10$ & $107$ & $63.493001$ \\
$2$ & $10$ & $108$ & $0.0007963$ \\
$2$ & $10$ & $109$ & $0.0000256$ \\
$2$ & $10$ & $110$ & $0.0000001$ \\
\hline
$3$ & $1$ & $84$ & $40.187937$ \\
$3$ & $1$ & $85$ & $2.1926074$ \\
$3$ & $1$ & $86$ & $0.0890408$ \\
$3$ & $1$ & $87$ & $0.0028106$ \\
$3$ & $1$ & $88$ & $0.0002627$ \\
$3$ & $1$ & $89$ & $0.0001896$ \\
\hline
$3$ & $2$ & $88$ & $243.40542$ \\
$3$ & $2$ & $89$ & $6.2933477$ \\
$3$ & $2$ & $90$ & $0.1099001$ \\
$3$ & $2$ & $91$ & $0.0013993$ \\
$3$ & $2$ & $92$ & $0.0001487$ \\
\hline
$3$ & $3$ & $92$ & $131.08327$ \\
$3$ & $3$ & $93$ & $1.1823671$ \\
$3$ & $3$ & $94$ & $0.0065480$ \\
$3$ & $3$ & $95$ & $0.0001237$ \\
$3$ & $3$ & $96$ & $0.0000938$ \\
\hline
\end{tabular}
\begin{tabular}{| c | c | c | l |}
\hline
$\alpha_1$ & $\alpha$ & $Y_0$ & $C$ \\
\hline
$3$ & $4$ & $95$ & $557.01746$ \\
$3$ & $4$ & $96$ & $2.1123032$ \\
$3$ & $4$ & $97$ & $0.0044867$ \\
$3$ & $4$ & $98$ & $0.0000815$ \\
\hline
$3$ & $5$ & $98$ & $297.55254$ \\
$3$ & $5$ & $99$ & $0.3674601$ \\
$3$ & $5$ & $100$ & $0.0002844$ \\
$3$ & $5$ & $101$ & $0.0000569$ \\
\hline
$3$ & $6$ & $101$ & $10.558512$ \\
$3$ & $6$ & $102$ & $0.0031396$ \\
$3$ & $6$ & $103$ & $0.0000470$ \\
\hline
$3$ & $7$ & $103$ & $89.615782$ \\
$3$ & $7$ & $104$ & $0.0109304$ \\
$3$ & $7$ & $105$ & $0.0000386$ \\
\hline
$3$ & $8$ & $105$ & $183.49683$ \\
$3$ & $8$ & $106$ & $0.0078087$ \\
$3$ & $8$ & $107$ & $0.0000313$ \\
\hline
$3$ & $9$ & $107$ & $70.410044$ \\
$3$ & $9$ & $108$ & $0.0008799$ \\
$3$ & $9$ & $109$ & $0.0000256$ \\
\hline
$3$ & $10$ & $109$ & $3.7235338$ \\
$3$ & $10$ & $110$ & $0.0000334$ \\
$3$ & $10$ & $111$ & $0.0000003$ \\
$3$ & $10$ & $112$ & $0.0000001$ \\
\hline
$4$ & $1$ & $88$ & $279.89541$ \\
$4$ & $1$ & $89$ & $7.2338934$ \\
$4$ & $1$ & $90$ & $0.1298417$ \\
$4$ & $1$ & $91$ & $0.0017720$ \\
$4$ & $1$ & $92$ & $0.0001558$ \\
\hline
$4$ & $2$ & $92$ & $149.37762$ \\
$4$ & $2$ & $93$ & $1.3500825$ \\
$4$ & $2$ & $94$ & $0.0075153$ \\
$4$ & $2$ & $95$ & $0.0001277$ \\
$4$ & $2$ & $96$ & $0.0000938$ \\
\hline
$4$ & $3$ & $95$ & $629.69393$ \\
$4$ & $3$ & $96$ & $2.3881358$ \\
$4$ & $3$ & $97$ & $0.0050654$ \\
$4$ & $3$ & $98$ & $0.0000822$ \\
\hline
\end{tabular}
\end{small}
\end{center}
\end{table}

\begin{table}[t]
\begin{center}
\begin{small}
\begin{tabular}{| c | c | c | l |}
\hline
$\alpha_1$ & $\alpha$ & $Y_0$ & $C$ \\
\hline
$4$ & $4$ & $98$ & $333.97042$ \\
$4$ & $4$ & $99$ & $0.4124195$ \\
$4$ & $4$ & $100$ & $0.0003115$ \\
$4$ & $4$ & $101$ & $0.0000569$ \\
\hline
$4$ & $5$ & $101$ & $11.775583$ \\
$4$ & $5$ & $102$ & $0.0034955$ \\
$4$ & $5$ & $103$ & $0.0000470$ \\
\hline
$4$ & $6$ & $103$ & $99.380273$ \\
$4$ & $6$ & $104$ & $0.0121164$ \\
$4$ & $6$ & $105$ & $0.0000387$ \\
\hline
$4$ & $7$ & $105$ & $202.45206$ \\
$4$ & $7$ & $106$ & $0.0086116$ \\
$4$ & $7$ & $107$ & $0.0000314$ \\
\hline
$4$ & $8$ & $107$ & $77.323788$ \\
$4$ & $8$ & $108$ & $0.0009635$ \\
$4$ & $8$ & $109$ & $0.0000256$ \\
\hline
$4$ & $9$ & $109$ & $4.0719188$ \\
$4$ & $9$ & $110$ & $0.0000343$ \\
\hline
$4$ & $10$ & $110$ & $10553.818$ \\
$4$ & $10$ & $111$ & $0.0204677$ \\
$4$ & $10$ & $112$ & $0.0000190$ \\
$4$ & $10$ & $113$ & $0.0000001$ \\
\hline
$5$ & $1$ & $92$ & $167.66332$ \\
$5$ & $1$ & $93$ & $1.5120638$ \\
$5$ & $1$ & $94$ & $0.0100915$ \\
$5$ & $1$ & $95$ & $0.0001509$ \\
$5$ & $1$ & $96$ & $0.0000939$ \\
\hline
$5$ & $2$ & $95$ & $702.27252$ \\
$5$ & $2$ & $96$ & $2.6767555$ \\
$5$ & $2$ & $97$ & $0.0058542$ \\
$5$ & $2$ & $98$ & $0.0000850$ \\
\hline
$5$ & $3$ & $98$ & $370.41678$ \\
$5$ & $3$ & $99$ & $0.4578772$ \\
$5$ & $3$ & $100$ & $0.0003417$ \\
$5$ & $3$ & $101$ & $0.0000570$ \\
\hline
\end{tabular}
\begin{tabular}{| c | c | c | l |}
\hline
$\alpha_1$ & $\alpha$ & $Y_0$ & $C$ \\
\hline
$5$ & $4$ & $101$ & $12.992382$ \\
$5$ & $4$ & $102$ & $0.0038532$ \\
$5$ & $4$ & $103$ & $0.0000470$ \\
\hline
$5$ & $5$ & $103$ & $109.13817$ \\
$5$ & $5$ & $104$ & $0.0133020$ \\
$5$ & $5$ & $105$ & $0.0000388$ \\
\hline
$5$ & $6$ & $105$ & $221.39416$ \\
$5$ & $6$ & $106$ & $0.0094139$ \\
$5$ & $6$ & $107$ & $0.0000314$ \\
\hline
$5$ & $7$ & $107$ & $84.233287$ \\
$5$ & $7$ & $108$ & $0.0010471$ \\
$5$ & $7$ & $109$ & $0.0000256$ \\
\hline
$5$ & $8$ & $109$ & $4.4201216$ \\
$5$ & $8$ & $110$ & $0.0000353$ \\
\hline
$5$ & $9$ & $110$ & $11419.253$ \\
$5$ & $9$ & $111$ & $0.0221440$ \\
$5$ & $9$ & $112$ & $0.0000190$ \\
\hline
$5$ & $10$ & $112$ & $22.759465$ \\
$5$ & $10$ & $113$ & $0.0000240$ \\
$5$ & $10$ & $114$ & $0.0000001$ \\
\hline
$6$ & $1$ & $95$ & $774.84324$ \\
$6$ & $1$ & $96$ & $2.9379648$ \\
$6$ & $1$ & $97$ & $0.0062067$ \\
$6$ & $1$ & $98$ & $0.0000955$ \\
\hline
$6$ & $2$ & $98$ & $406.74243$ \\
$6$ & $2$ & $99$ & $0.5177958$ \\
$6$ & $2$ & $100$ & $0.0005074$ \\
$6$ & $2$ & $101$ & $0.0000572$ \\
\hline
$6$ & $3$ & $101$ & $14.230994$ \\
$6$ & $3$ & $102$ & $0.0043119$ \\
$6$ & $3$ & $103$ & $0.0000473$ \\
\hline
$6$ & $4$ & $103$ & $118.90823$ \\
$6$ & $4$ & $104$ & $0.0145133$ \\
$6$ & $4$ & $105$ & $0.0000389$ \\
\hline
\end{tabular}
\end{small}
\end{center}
\end{table}

\begin{table}[t]
\begin{center}
\begin{small}
\begin{tabular}{| c | c | c | l |}
\hline
$\alpha_1$ & $\alpha$ & $Y_0$ & $C$ \\
\hline
$6$ & $5$ & $105$ & $240.33821$ \\
$6$ & $5$ & $106$ & $0.0102188$ \\
$6$ & $5$ & $107$ & $0.0000314$ \\
\hline
$6$ & $6$ & $107$ & $91.142884$ \\
$6$ & $6$ & $108$ & $0.0011307$ \\
$6$ & $6$ & $109$ & $0.0000256$ \\
\hline
$6$ & $7$ & $109$ & $4.7683260$ \\
$6$ & $7$ & $110$ & $0.0000362$ \\
\hline
$6$ & $8$ & $110$ & $12284.687$ \\
$6$ & $8$ & $111$ & $0.0238203$ \\
$6$ & $8$ & $112$ & $0.0000190$ \\
\hline
$6$ & $9$ & $112$ & $24.420945$ \\
$6$ & $9$ & $113$ & $0.0000245$ \\
\hline
$6$ & $10$ & $113$ & $14130.221$ \\
$6$ & $10$ & $114$ & $0.0024791$ \\
$6$ & $10$ & $115$ & $0.0000044$ \\
$6$ & $10$ & $116$ & $0.0000001$ \\
\hline
$7$ & $1$ & $98$ & $443.12389$ \\
$7$ & $1$ & $99$ & $0.5471351$ \\
$7$ & $1$ & $100$ & $0.0003927$ \\
$7$ & $1$ & $101$ & $0.0000574$ \\
\hline
$7$ & $2$ & $101$ & $15.423819$ \\
$7$ & $2$ & $102$ & $0.0082841$ \\
$7$ & $2$ & $103$ & $0.0000503$ \\
\hline
$7$ & $3$ & $103$ & $129.05908$ \\
$7$ & $3$ & $104$ & $0.0169840$ \\
$7$ & $3$ & $105$ & $0.0000402$ \\
\hline
$7$ & $4$ & $105$ & $259.38714$ \\
$7$ & $4$ & $106$ & $0.0111522$ \\
$7$ & $4$ & $107$ & $0.0000315$ \\
\hline
$7$ & $5$ & $107$ & $98.059267$ \\
$7$ & $5$ & $108$ & $0.0012172$ \\
$7$ & $5$ & $109$ & $0.0000256$ \\
\hline
$7$ & $6$ & $109$ & $5.1166158$ \\
$7$ & $6$ & $110$ & $0.0000372$ \\
\hline
\end{tabular}
\begin{tabular}{| c | c | c | l |}
\hline
$\alpha_1$ & $\alpha$ & $Y_0$ & $C$ \\
\hline
$7$ & $7$ & $110$ & $13150.105$ \\
$7$ & $7$ & $111$ & $0.0254970$ \\
$7$ & $7$ & $112$ & $0.0000190$ \\
\hline
$7$ & $8$ & $112$ & $26.082288$ \\
$7$ & $8$ & $113$ & $0.0000250$ \\
\hline
$7$ & $9$ & $113$ & $15059.773$ \\
$7$ & $9$ & $114$ & $0.0026411$ \\
\hline
$7$ & $10$ & $115$ & $0.4569526$ \\
$7$ & $10$ & $116$ & $0.0000127$ \\
$7$ & $10$ & $117$ & $0.0000001$ \\
\hline
$8$ & $1$ & $101$ & $16.639402$ \\
$8$ & $1$ & $102$ & $0.0049176$ \\
$8$ & $1$ & $103$ & $0.0000478$ \\
\hline
$8$ & $2$ & $103$ & $138.40436$ \\
$8$ & $2$ & $104$ & $0.0168563$ \\
$8$ & $2$ & $105$ & $0.0000514$ \\
\hline
$8$ & $3$ & $105$ & $278.20999$ \\
$8$ & $3$ & $106$ & $0.0176009$ \\
$8$ & $3$ & $107$ & $0.0000329$ \\
\hline
$8$ & $4$ & $107$ & $105.25995$ \\
$8$ & $4$ & $108$ & $0.0015069$ \\
$8$ & $4$ & $109$ & $0.0000256$ \\
\hline
$8$ & $5$ & $109$ & $5.4707339$ \\
$8$ & $5$ & $110$ & $0.0000394$ \\
\hline
$8$ & $6$ & $110$ & $14015.192$ \\
$8$ & $6$ & $111$ & $0.0271902$ \\
$8$ & $6$ & $112$ & $0.0000190$ \\
\hline
$8$ & $7$ & $112$ & $27.743089$ \\
$8$ & $7$ & $113$ & $0.0000255$ \\
\hline
$8$ & $8$ & $113$ & $15988.974$ \\
$8$ & $8$ & $114$ & $0.0028032$ \\
\hline
$8$ & $9$ & $115$ & $0.4842973$ \\
$8$ & $9$ & $116$ & $0.0000127$ \\
\hline
$8$ & $10$ & $116$ & $42.515883$ \\
$8$ & $10$ & $117$ & $0.0000119$ \\
$8$ & $10$ & $118$ & $0.0000001$ \\
\hline
\end{tabular}
\end{small}
\end{center}
\end{table}

\begin{table}[t]
\begin{center}
\begin{small}
\begin{tabular}{| c | c | c | l |}
\hline
$\alpha_1$ & $\alpha$ & $Y_0$ & $C$ \\
\hline
$9$ & $1$ & $103$ & $148.15370$ \\
$9$ & $1$ & $104$ & $0.0180405$ \\
$9$ & $1$ & $105$ & $0.0000390$ \\
\hline
$9$ & $2$ & $105$ & $297.13835$ \\
$9$ & $2$ & $106$ & $0.0126220$ \\
$9$ & $2$ & $107$ & $0.0000332$ \\
\hline
$9$ & $3$ & $107$ & $111.86299$ \\
$9$ & $3$ & $108$ & $0.0043396$ \\
$9$ & $3$ & $109$ & $0.0000263$ \\
\hline
$9$ & $4$ & $109$ & $6.0873692$ \\
$9$ & $4$ & $110$ & $0.0000896$ \\
\hline
$9$ & $5$ & $110$ & $14886.409$ \\
$9$ & $5$ & $111$ & $0.0302231$ \\
$9$ & $5$ & $112$ & $0.0000190$ \\
\hline
$9$ & $6$ & $112$ & $29.418754$ \\
$9$ & $6$ & $113$ & $0.0000267$ \\
\hline
$9$ & $7$ & $113$ & $16918.289$ \\
$9$ & $7$ & $114$ & $0.0029683$ \\
\hline
$9$ & $8$ & $115$ & $0.5116813$ \\
$9$ & $8$ & $116$ & $0.0000127$ \\
\hline
$9$ & $9$ & $116$ & $44.850504$ \\
$9$ & $9$ & $117$ & $0.0000119$ \\
\hline
$9$ & $10$ & $117$ & $1923.3233$ \\
$9$ & $10$ & $118$ & $0.0000167$ \\
$9$ & $10$ & $119$ & $0.0000001$ \\
\hline
\end{tabular}
\begin{tabular}{| c | c | c | l |}
\hline
$\alpha_1$ & $\alpha$ & $Y_0$ & $C$ \\
\hline
$10$ & $1$ & $105$ & $316.06672$ \\
$10$ & $1$ & $106$ & $0.0134237$ \\
$10$ & $1$ & $107$ & $0.0000314$ \\
\hline
$10$ & $2$ & $107$ & $118.76799$ \\
$10$ & $2$ & $108$ & $0.0014647$ \\
$10$ & $2$ & $109$ & $0.0000257$ \\
\hline
$10$ & $3$ & $109$ & $6.1605138$ \\
$10$ & $3$ & $110$ & $0.0002801$ \\
\hline
$10$ & $4$ & $110$ & $15744.953$ \\
$10$ & $4$ & $111$ & $0.0946919$ \\
$10$ & $4$ & $112$ & $0.0000209$ \\
\hline
$10$ & $5$ & $112$ & $31.907392$ \\
$10$ & $5$ & $113$ & $0.0000663$ \\
\hline
$10$ & $6$ & $113$ & $17854.358$ \\
$10$ & $6$ & $114$ & $0.0034270$ \\
\hline
$10$ & $7$ & $115$ & $0.5399268$ \\
$10$ & $7$ & $116$ & $0.0000128$ \\
\hline
$10$ & $8$ & $116$ & $47.186782$ \\
$10$ & $8$ & $117$ & $0.0000120$ \\
\hline
$10$ & $9$ & $117$ & $2020.8916$ \\
$10$ & $9$ & $118$ & $0.0000171$ \\
\hline
$10$ & $10$ & $118$ & $40634.957$ \\
$10$ & $10$ & $119$ & $0.0000610$ \\
$10$ & $10$ & $120$ & $0.0000001$ \\
\hline
\end{tabular}
\end{small}
\end{center}
\end{table}

\section{Proof of Thorem \ref{thm13}}
We see that
\begin{equation}
\abs{\psi(x, \chi)-\psi(x, \chi^*)}\leq\sum_{p^k\leq x, p\mid q}\log p<\omega(q)(\log x)\leq\log^2(qx),
\end{equation}
where $\omega(q)$ denotes the number of distinct prime factors of $q$.  Noting that $\psi(x, \chi_0^*)=\psi(x)$,
we have
\begin{equation}
\abs{\psi(x, q, a)-\frac{\psi(x)}{\vph(q)}}\leq\frac{1}{\vph(q)}\abs{\sum_{\chi\pmod{q}, \chi\neq \chi_0}\psi(x, \chi^*)}+\log^2(qx).
\end{equation}

Hence we obtain
\begin{equation}\label{eq60}
\begin{split}
\sum_{q\leq Q}\max_{a\pmod{q}}\abs{\psi(x, q, a)-\frac{\psi(x)}{\vph(q)}} \leq & Q\log^2 (Qx)+\sum_{1<q\leq Q, q_0\nmid q}\frac{1}{\vph(q)}\abs{\sum_{\chi\pmod{q}, \chi\neq \chi_0}\psi(x, \chi^*)} \\
< & x^\frac{1}{2}+\sum_{1<q\leq Q, q_0\nmid q}\frac{1}{\vph(q)}\abs{\sum_{\chi\pmod{q}, \chi\neq \chi_0}\psi(x, \chi^*)}.
\end{split}
\end{equation}
Since each character $\chi\pmod{q}$ is induced from a certain primitive character $\chi^*\pmod{q^*}$ with $q^*\mid q$,
the last sum is at most
\begin{equation}\label{eq61}
\begin{split}
& \sum_{\substack{m\geq 1, q^*>1,\\ q_0\nmid q^*,\\ mq^*\leq Q}}\frac{1}{\vph(mq^*)}\abs{\sum_{\chi\pmod{q^*}}^*\psi(x, \chi)}\\
\leq & \left(\sum_{1\leq m\leq Q}\frac{1}{\vph(m)}\right)\sum_{1<q\leq Q, q_0\nmid q}\abs{\sideset{}{^*}\sum_{\chi\pmod{q}}\psi(x, \chi)}.
\end{split}
\end{equation}

Using Theorem 1.2 of \cite{AH}, we have
\begin{equation}\label{eq62}
\begin{split}
\sum_{Q_1<q\leq Q}\frac{1}{\vph(q)}\abs{\sideset{}{^*}\sum_{\chi\pmod{q}}\psi(x, \chi)}\leq & c_0\left(\frac{4x}{Q_1}+4x^\frac{1}{2}Q+18x^\frac{2}{3}Q^\frac{1}{2}+\frac{5}{2}x^\frac{5}{6}\log x\right)\log^\frac{7}{2} x\\
\leq & c_0\left(\frac{4x}{Q_1}+\frac{4x}{\log^A x}+\frac{18x^\frac{11}{12}}{\log^\frac{A}{2} x}+\frac{5}{2}x^\frac{5}{6}\log x \right)\log^\frac{7}{2} x\\
\leq & \frac{4c_0 x}{Q_1}+\frac{c_0 x}{\log^{A-\frac{7}{2}} x}\left(4+\frac{18\log^\frac{A}{2} x}{x^\frac{1}{12}}+\frac{5\log^{A+1} x}{2x^\frac{1}{6}} \right)\\
< & \frac{4c_0 x}{Q_1}+\frac{c_0(4+e^{-800}) x}{\log^{A_1-\frac{7}{2}} x},
\end{split}
\end{equation}
where
\begin{equation}
\begin{split}
c_0= & \frac{2^{\frac{13}{2}}}{9\pi\log 2}\left(\frac{1}{3}+\frac{3}{2\log 2}\right)\left(\frac{2+\log(\log 2/\log(4/3))}{\log 2}\right)\sqrt\frac{\psi(113)}{113}\\
< & 48.833.
\end{split}
\end{equation}

Now we consider the sum $\sideset{}{^*}\sum_{\chi\pmod{q}}\psi(x, \chi)$ for moduli $q$ with $q\leq Q_1$, $q_0\nmid q$.

The proof of Theorem 3.6 of \cite{MC2} shows that the left-hand side quantity $\frac{\vph(q)}{x}\abs{\psi(x; q, a)-\frac{x}{\vph(q)}}$
in this theorem can be replaced by $-1+x^{-1}\sum_{\chi\pmod{q}}\abs{\psi(x, \chi)}$.  This also applies to Theorem \ref{thm11}.

Now, provided that $x\geq x_0$, Theorem \ref{thm11} gives
\begin{equation}\label{eq63}
\abs{-1+\frac{1}{x}\sum_{\chi\pmod{q}}\abs{\psi(x, \chi)}}<\frac{C_0}{\log^{A+1} x}+E_0\frac{x^{1-\beta_0}}{\beta_0},
\end{equation}
where $E_0=1$ and $\beta_0$ denote the Siegel zero modulo $q$ if it exists and $E_0=0$ otherwise.

We see that
\begin{equation}
\begin{split}
\sideset{}{^*}\sum_{\chi\pmod{q}}\abs{\psi(x, \chi)}\leq & \abs{-1+\frac{1}{x}\sum_{\chi\pmod{q}}\abs{\psi(x, \chi)}}+\abs{\frac{\psi(x, \chi_0)}{x}-1}\\
& \leq \abs{-1+\frac{1}{x}\sum_{\chi\pmod{q}}\abs{\psi(x, \chi)}}+\abs{\frac{\psi(x)}{x}-1}+\frac{\log q}{x\log 2}\\
& \leq \abs{-1+\frac{1}{x}\sum_{\chi\pmod{q}}\abs{\psi(x, \chi)}}+\frac{e^{-18} x}{\log^{A-3} x},
\end{split}
\end{equation}
using the inequality (Theorem 11 of \cite{Sch})
\begin{equation}
\abs{\frac{\psi(x)}{x}-1}<\sqrt{\frac{8}{17\pi}}\left(\frac{\log x}{R_2}\right)^\frac{1}{4}\exp -\sqrt{\frac{\log x}{R_2}}
\end{equation}
for $x\geq 101$, where $R_2=9.645908801$.
Therefore (\ref{eq63}) becomes
\begin{equation}\label{eq64}
\sideset{}{^*}\sum_{\chi\pmod{q}}\abs{\psi(x, \chi)}<\frac{(C_0+e^{-18})x}{\log^{A-3} x}+E_0\frac{x^{1-\beta_0}}{1-\beta_0},
\end{equation}
where $E_0=1$ and $\beta_0$ denote the Siegel zero modulo $q$ if it exists and $E_0=0$ otherwise.

Since $q\leq Q_1$ is non-exceptional, the right-hand side of (\ref{eq64}) is at most
\begin{equation}
\frac{(C_0+e^{-18})x}{\log^{A-3} x}+\frac{x^{1-\frac{1}{2AR_1\log\log x}}}{1-\frac{1}{2AR_1\log\log x}}.
\end{equation}
We can easily confirm that, provided that $x\geq x_0$,
\begin{equation}
\frac{x^{1-\frac{1}{2AR_1\log\log x}}}{1-\frac{1}{2AR_1\log\log x}}<e^{-50}\frac{x}{\log^{A-3} x}
\end{equation}
and therefore
\begin{equation}
\sideset{}{^*}\sum_{\chi\pmod{q}}\abs{\psi(x, \chi)}<\frac{(C_0+e^{-17})x}{\log^{A-3} x}.
\end{equation}
We see that $\sum_{1\leq m\leq Q_1}\frac{1}{\vph(m)}<c_1(1+\log Q_1)$ from the argument
in the proof of Theorem A. 17 in \cite[p. 316]{Nat}.
Thus we obtain
\begin{equation}\label{eq65}
\sum_{q\leq Q_1, q_0\nmid q}\frac{1}{\vph(q)}\sideset{}{^*}\sum_{\chi\pmod{q}}\abs{\psi(x, \chi)}\leq \frac{c_1(C_0+e^{-17})x(1+A\log\log x)}{\log^{A-3} x}.
\end{equation}

Observing that $\sum_{m\leq Q}\frac{1}{\vph(m)}<c_1(1+\log Q)<\frac{c_1}{2}\log x$
and substituting (\ref{eq62}), (\ref{eq65}) into (\ref{eq61}), we see that the last sum of (\ref{eq60}) is at most
\begin{equation}
\frac{c_1c_0(2+e^{-800})x}{\log^{A-\frac{9}{2}} x}+\frac{2c_0c_1x\log^\frac{9}{2} x}{Q_1}+\frac{c_1^2(C_0+e^{-17})x(1+A\log\log x)}{2\log^{A-4} x}.
\end{equation}
This proves Theorem \ref{thm13}.

{}
\vskip 12pt

{\small Tomohiro Yamada}\\
{\small Center for Japanese language and culture\\Osaka University\\562-8558\\8-1-1, Aomatanihigashi, Minoo, Osaka\\Japan}\\
{\small e-mail: \href{mailto:tyamada1093@gmail.com}{tyamada1093@gmail.com}}

\begin{thebibliography}{}
\bibitem{AH}
Amir Akbary and Kyle Hambrook, A variant of the Bombieri-Vinogradov theorem with explicit constants and applications, Math. Comp. (to appear).
\bibitem{Bom}
E. Bombieri, On the large sieve, {\it Mathematika} \textbf{12} (1965), 201--225.
\bibitem{CW}
Chen Jing-Run and Wang Tian-Ze, On distribution of primes in an arithmetical progression, {\it Sci. Sinica Ser. A} \textbf{33} (1990), 397--408.
\bibitem{Dus}
P. Dusart, Estimates of $\theta(x; k, l)$ for large values of $x$, {\it Math. Comp.} \textbf{71} (2002), 1137--1166.
\bibitem{LW}
Ming-Chit Liu and Tianze Wang, Distribution of zeros of Dirichlet $L$-functions and an explicit formula for $\psi(t, \chi)$,
{\it Acta Arith}. \textbf{102} (2002), 261--293.
\bibitem{Kad}
Habiba Kadiri, An explicit zero-free region for the Dirichlet L-functions, 
\url{http://arxiv.org/abs/0510570}.
\bibitem{MC1}
Kevin S. McCurley, Explicit zero-free regions for Dirichlet $L$-functions, 
{\it J. Number Theory} \textbf{19} (1984), 7--32.
\bibitem{MC2}
Kevin S. McCurley, Explicit estimates for the error term in the prime number theorem for arithmetic progressions,
{\it Math. Comp.} \textbf{42} (1985), 265--285.
\bibitem{MC3}
Kevin S. McCurley, Explicit estimates for $\theta(x; 3, 1)$ and $\psi(x; 3, 1)$,
{\it Math. Comp.} \textbf{42} (1985), 287--296.
\bibitem{Nat}
Melvyn B. Nathanson, {\it Additive Number Theory: The Classical Bases}, GTM 164, Springer-Verlag, New York, 1996.
\bibitem{RR}
O. Ramar\'{e} and R. Rumely, Primes in arithmetic progressions, {\it Math. Comp.} \textbf{65} (1996), 397--425.
\bibitem{Ros}
Barkley Rosser, Explicit bounds for some functions of prime numbers, {\it Amer. J. Math.} \textbf{63} (1941), 211--232.
\bibitem{RS1}
J. Barkley Rosser and Lowell Schoenfeld, Approximate formulas for some functions of prime numbers,
{\it Illinois J. Math.} \textbf{6} (1962), 64--94.
\bibitem{RS2}
J. Barkley Rosser and Lowell Schoenfeld, Sharper Bounds for the Chebyshev Functions $\theta(x)$ and $\psi(x)$,
{\it Math. Comp} \textbf{29} (1975), 243--269.
\bibitem{Rum}
R. Rumely, Numerical computations concerning the ERH, {\it Math. Comp.} \textbf{61} (1993), 415--440.
\bibitem{Sch}
Lowell Schoenfeld, Sharper Bounds for the Chebyshev Functions $\theta(x)$ and $\psi(x)$ II,
{\it Math. Comp} \textbf{30} (1976), 337--360.
\bibitem{Tru}
T. S. Trudgian, An improved upper bound for the error in the zero-counting formulae for
Dirichlet $L$-functions and Dedekind zeta-functions, (submitted).
\bibitem{Ymd1}
Tomohiro Yamada, An explicit formula for the linear sieve, (in preparation).
\bibitem{Ymd2}
Tomohiro Yamada, Explicit Chen's theorem, (in preparation).
\end{thebibliography}
\end{document}